\documentclass[preprint,12 pt]{elsarticle}
\usepackage{}

\usepackage{amsfonts}
\usepackage{amsmath}
\usepackage{amssymb}
\usepackage{graphics}
\usepackage{graphicx}
\usepackage{mathrsfs}
\usepackage{indentfirst}
\usepackage{stmaryrd}
\usepackage{latexsym}
\usepackage{xypic}
\newtheorem{definition}{Definition}[section]

\newtheorem{theorem}[definition]{Theorem}
\newtheorem{proposition}[definition]{Proposition}
\newtheorem{example}[definition]{Example}

\newtheorem{remark}[definition]{Remark}
\newtheorem{corollary}[definition]{Corollary}
\newproof{proof}{\textbf{Proof}}

\journal{Iranian Journal of Fuzzy Systems}

\begin{document}

\begin{frontmatter}



\title{\textbf{Monadic NM-algebras}}


\author{Jun Tao Wang$^a$,~Xiao Long Xin$^{a,*}$,Peng Fei He$^b$}
\cortext[cor1]{Corresponding author. \\
Email addresses: wjt@stumail.nwu.edu.cn(J.T. Wang), xlxin@nwu.edu.cn (X.L. Xin),\\ hepengf1986@126.com (P. F. He)}

\address[A]{School of Mathematics, Northwest University, Xi'an, 710127, China}
\address[B]{School of Mathematics and Information Science, Shaanxi Normal University, Xi'an, 710119, China}

\begin{abstract}
In this paper, we introduce and investigate monadic NM-algebras: a variety of NM-algebras equipped with  universal quantifiers. Also, we obtain some conditions under which monadic NM-algebras become  monadic Boolean algebras. Besides, we show that the variety of monadic NM-algebras faithfully the axioms on quantifiers in monadic predicate NM logic. Furthermore, we discuss relations between monadic NM-algebras and  some related structures, likeness modal NM-algebras and rough approximation spaces. In addition, we investigate monadic filters in monadic NM-algebras. In particular, we characterize simple and subdirectly irreducible monadic NM-algebras and obtain a representation theorem for  monadic NM-algebras. Finally, we present monadic NM-logic and prove the (chain) completeness of monadic NM-logic based on monadic NM-algebras. These results constitute a crucial first step for providing a solid algebraic foundation for the monadic predicate NM logic.
\end{abstract}


\end{frontmatter}

\section{Introduction}
\label{intro}It is well known that non-classical logics take the advantage of the classical logics to handle uncertain information and fuzzy information. While Boolean algebras are algebraic semantics for classical logics, many-valued logical algebras serve as algebraic semantics for non-classical logics. Until now, various kinds of many-valued logical algebras have been extensively introduced
and studied, for example, MV-algebras, BL-algebras,
G\"{o}del algebras, MTL-algebras and NM-algebras. Among these many-valued logical algebras, MTL-algebras are the most significant because the others are
all particular
cases of MTL-algebras. In fact, MTL-algebras contain all algebras
induced by left continuous t-norm and their residua. As an most
 important class of MTL-algebras,
nilpotent minimum algebras (NM-algebras for short) are MTL-algebras
 satisfying the low of involution of the negation $\neg\neg x=x$
and $(\neg(x\odot y))\vee(x\wedge y)\rightarrow (x\odot y))=1$, are
the corresponding algebraic structures of nilpotent minimum
 logic \cite{Bianchi}.
 Well known many-valued logic such as  {\L}ukasiewicz logic, G\"{o}del
logic and product logic can be regarded as schematic extensions of
 basic logic
(BL for short), which is a general framework of many-valued logic for
capturing the tautologies of continuous t-norm and their residua
\cite{Hajek}.
 It must be pointed out here that the NM logic is not an axiomatic
 extension of BL because the corresponding t-norm in the NM-algebra
 is not continuous but
only left-continuous.
  Viewing the axioms of NM-algebras, we note that NM-algebras
are different from BL-algebras, since they does not satisfy
the divisibility condition $x\wedge y = x\odot(x\rightarrow y)$.

Monadic (Boolean) algebras in the sense of Halmos \cite{Halmos}
 are Boolean algebras equipped with a closure operator $\exists$ whose range is
a subalgebra of Boolean algebra. This operator abstracts algebraic properties of the standard existential quantifier ``for some". The name ``monadic" comes
from the connection with predicate logics for languages having one placed predicates and a single quantifier. Monadic Boolean algebras have been deeply
investigated in \cite{Henkin, Nemeti}. Inspired by this,  monadic Heyting algebras, an algebraic model of the one-variable
 fragment of the intuitionistic predicate logic, were introduced and developed in \cite{ Bezhanishvili, Monteiro 1,Monteiro 2}. Subsequently, monadic MV-algebras, an algebraic model of the one element fragment of {\L}ukasiewicz predicate
logic, were introduced and studied in \cite{Nola,Rutledge,Rachunek2}.  After then, monadic BL-algebras,
monadic residuated lattices, monadic basic algebras and monadic residuated $\ell$-monoids were introduced and investigated in
\cite{Grigolia,Rachunek1,Chajda,Rachunek3}.
 In the above-mentioned monadic algebras, since both MV-algebras and basic algebras satisfy De Morgan and double negation laws, in the definition
 of the corresponding monadic algebras, it is possible to use only one of the existential and universal quantifiers as initial, the other is then
definable as the dual of the original one. On the contrary, the definitions of monadic Heyting algebras, monadic BL-algebras and monadic residuated
lattices require using both kinds of quantifiers simultaneously, because these quantifiers are not mutually interdefinable.

In this paper, we will investigate monadic NM-algebras  and prove the (chain) completeness of the monadic NM-logics.
 One of our aims is to introduce the variety of NM-algebras endowed with universal quantifiers. In particular, the paper \cite{Bianchi} was the first
attempt to define monadic NM-algebras using the existential and universal quantifiers (analogously as for monadic BL-algebras). But it seems to be more
appropriate to introduce such monadic algebras similarly as the monadic MV-algebras since NM-algebras also satisfy De Morgan and double negation laws.
On the other hand, the main focus of existing research about quantifiers is on MV-algebras, BL-algebras, Heyting algebras, basic algebras and residuated $\ell$-monoids,
 etc. All the above-mentioned algebraic structures satisfy the divisibility condition. In this case, the conjunction $\odot$ on the unit interval corresponds
to a continuous t-norm.  However, there are few research about quantifiers on residuated structures without the divisibility condition so far. In fact,
NM-algebras are the residuated structure, which does not satisfy the divisibility condition. Therefore, it is interesting to study monadic NM-algebras.
These are motivations for us to investigate monadic NM-algebras.

This paper is organized as follows: In Section 2, we review some basic definitions and results about NM-algebras. In Section 3, we introduce the notion of
monadic NM-algebras and investigate some related properties of them. Also, we give some conditions under which monadic NM-algebras become  monadic Boolean
algebras and discuss relations between monadic NM-algebras and related structures. In Section 4, we characterize simple and subdirectly irreducible monadic
NM-algebras and prove a representation theorem for monadic NM-algebras. In Section 5, we present an axiom system of monadic NM-logic and prove the (chain)
completeness of this logic.

\section{Preliminaries}
In this section, we summarize some definitions and results about NM-algebras, which will be used in the following sections.
\begin{definition}\emph{\cite{Esteva} An algebraic structure $(L,\wedge,\vee,\odot,\rightarrow,0,1)$ of type $(2,2,2,2,0,0)$ is called an \emph{NM-algebra} if it satisfies the following conditions:
\begin{enumerate}[(1)]
\item $(L,\wedge,\vee,0,1)$ is a bounded lattice,
\item $(L,\odot,1)$ is a commutative monoid,
\item $x\odot y\leq z$ if and only if $x\leq y\rightarrow z$,
\item $(x\rightarrow y)\vee(y\rightarrow x)=1$,
\item $(x\odot y\rightarrow 0)\vee(x\wedge y\rightarrow x\odot y)=1$,
\item $(x\rightarrow 0)\rightarrow 0=x$,
\end{enumerate}
for any $x,y,z\in L$.}
\end{definition}

In what follows, by $L$ we denote the universer of
an NM-algebra $(L,\wedge,\vee,\odot,\rightarrow,0,1)$.
For any $x,y\in L$, we define $\neg x=x\rightarrow 0$,
$\neg\neg x=\neg(\neg x)$, $x\odot y=\neg( x\rightarrow \neg y)$,
$x\oplus y=\neg x\rightarrow y$, $x^0=1$ and $x^n = x\odot x\cdots \odot x$
for any natural number $n\in N$. It is proved that
$x\oplus y=\neg(\neg x\odot \neg y)$.

\begin{proposition}\emph{\cite{Esteva,Wang1,Zhou}} In any NM-algebra $L$, the following properties hold: for all $x,y,z\in L$,
\begin{enumerate}[(1)]
               \item $x\leq y$ if and only if $x\rightarrow y=1$,
               \item $x\leq y\rightarrow x$,
               \item $x\leq y$ implies $y\rightarrow z\leq x\rightarrow z$,
               \item $x\leq y$ implies $z\rightarrow x\leq z\rightarrow y$,
               \item $x\vee y=((x\rightarrow y)\rightarrow y)\wedge((y\rightarrow x)\rightarrow x)$,
               \item $x\odot \neg x=0$ and $x\oplus \neg x=1$,
               \item $(x\odot y)\rightarrow z=x\rightarrow(y\rightarrow z)$,
               \item $x\rightarrow y=x\rightarrow(x\wedge y)$,
               \item $x\rightarrow (y\wedge z)=(x\rightarrow y)\wedge(x\rightarrow z)$,
               \item $(x\vee y)\rightarrow z=(x\rightarrow z)\wedge(y\rightarrow z)$,
               \item $(x\rightarrow y)^n\vee(y\rightarrow x)^n=1$,
               \item $(x\wedge y)\rightarrow z=(x\rightarrow z)\vee(y\rightarrow z)$.
             \end{enumerate}
\end{proposition}

We have the following characterizations for Boolean algebras.
\begin{theorem}\emph{\cite{Zhou}} Let $L$ be an NM-algebra. Then the following statements are equivalent:
\begin{enumerate}[(1)]
  \item $L$ is a Boolean algebra,
  \item $x\odot y=x\wedge y$ for any $x,y\in L$,
  \item $x\oplus y=x\vee y$ for any $x,y\in L$.
\end{enumerate}
\end{theorem}

 Let $L$ be an NM-algebra.
A nonempty subset $F$ of $L$ is called a \emph{filter} of $L$ if
 it satisfies: (1) $x,y\in F$ implies $x\odot y\in F$; (2) $x\in F$, $y\in L$ and $x\leq y$ implies $y\in F$.  A filter $F$ of $L$ is called a \emph{proper filter} if $F\neq L$. A proper filter $F$ of $L$ is called a \emph{maximal filter} if it is not contained in any proper filter of $L$. A proper filter $F$ of $L$ is called a \emph{prime filter} if for each $x,y\in L$ and $x\vee y\in F$, implies $x\in F$ or $y\in F$. A prime filter $F$ is said to be \emph{minimal} if $F$ is a minimal element in the set of prime filters of $L$ ordered by inclusion. For a nonempty subset $X$ of $L$, we denote by $\langle X\rangle$ is the filter generated by $X$. Clearly, we have $\langle X\rangle=\{x\in L|x\geq x_1\odot x_2\odot\cdots\odot x_n$, for some $n\in N$ and some $x_i\in X\}$. In particular, the principal filter generated by an element $x\in L$ is $\langle x\rangle=\{y\in L|y\geq x^n$, for some $n\in N\}$. If $F$ is a filter of $L$ and $x\in L$, then $\langle F\cup x\rangle=\{y\in L|y \geq f\odot x^n$, for some $f\in F, n\in N\rangle$. We denote by $F[L]$ be the set of all filers of $L$ and obatin that  $(F[L],\subseteq )$ forms a complete lattice \cite{Turunen,Esteva,Wang1,Zhang2}.

Let $F$ be a filter of an NM-algebra $L$. Define the \emph{congruence} $\equiv_F$ on $L$ by $x\equiv_Fy$ if and only if $x\rightarrow y\in F$ and $y\rightarrow x\in F$. The set of all congruence classes is denote by $L/F$, i.e. $L/F=\{[x]|x\in L\}$, where $[x]=\{x\in L|x\equiv_Fy\}$. Then $L/F$ becomes an NM-algebra with the natural operations induced from those of $L$. It is easily seen that if $F$ is a prime filter of $L$ if and only if $L/F$ is a linearly ordered NM-algebra \cite{Pei1,Wang1,Zhou}.


\begin{definition}\emph{\cite{Burris} Let $L$ be an NM-algebra. Then $L$ is called to be:}
\begin{enumerate}[(1)]
  \item \emph{\emph{simple} if it has exactly two filters: $\{1\}$ and $L$.}
  \item \emph{\emph{subdirectly irreducible} if among the nontrivial congruence of $L$ there exists the least one.}
\end{enumerate}
\end{definition}

 At the end of this section, we review the known main results about representation theory of NM-algebras, which is helpful for studying a representation theorem for monadic NM-algebras.

\begin{theorem}\emph{\cite{Zhang2}}  Let $L$ be an NM-algebra and $P$ be a filter of $L$. Then the following statements are equivalent:
\begin{enumerate}[(1)]
  \item $P$ is a minimal prime,
  \item $P=\cup\{a^\bot|a\in P\}$, where $a^\bot=\{x\in L|a\vee x=1\}$.
\end{enumerate}
\end{theorem}

\begin{definition}\emph{\cite{Burris} An NM-algebra $L$ is called \emph{representable} if $L$ is isomorphic to a subdirect product
of linearly ordered NM-algebras.}
\end{definition}

\begin{theorem}\emph{\cite{Zhang2}} Let $L$ be an NM-algebra. Then the following statements are equivalent:
\begin{enumerate}[(1)]
  \item $L$ is representable,
  \item there exists a set $S$ of prime filters such that $\bigcap S=\{1\}$.
\end{enumerate}
\end{theorem}

\section{Monadic NM-algebras}
In this section, we introduce and investigate monadic NM-algebras. Then, we prove that monadic NM-algebras faithfully the axioms on quantifiers in monadic predicate NM logic. Finally, we discuss the relation between monadic NM-algebras and  some related structures, likeness monadic Boolean algebras, modal NM-algebras and rough approximation spaces.
\begin{definition} \emph{Let $L$ be an NM-algebra. A mapping $\forall:L\rightarrow L$ is called a \emph{universal quantifier} on $L$, such that the following conditions are satisfied:
\begin{enumerate}[(U1)]
  \item  $\forall x\rightarrow x=1$,
  \item $\forall (\neg x\rightarrow \forall y)=\neg \forall x \rightarrow \forall y$,
  \item $\forall(\forall x\rightarrow y)=\forall x\rightarrow \forall y$,
  \item $\forall(x\vee \forall y)=\forall x\vee \forall y$.
\end{enumerate}
for any $x,y,z\in L$.}
\end{definition}

\begin{definition}\emph{Let $L$ be an NM-algebra and $\forall$ be a universal quantifier on $L$. Then the couple $(L,\forall)$ is called a \emph{monadic NM-algebra}.}
\end{definition}

\begin{remark}  \emph{(1) If $L$ is an NM-algebra, then it is an involution De Morgan algebra. From \cite {Ch}, we obtain that the quantifiers $\forall$ and $\exists$ on NM-algebras are interdefinable. So, we can define a unary operation $\exists x=\neg\forall\neg x$ corresponding to the existential quantifier. Then in any monadic NM-algebra hold the identities which are dual to (U1)-(U4):
\begin{enumerate}[(E1)]
  \item $x\rightarrow \exists x=1$,
  \item $\exists (\neg x\odot \exists y)=\exists \neg x\odot \exists y$,
  \item $\exists(\neg \exists x\odot \neg y)=\neg\exists x\odot \exists \neg y$,
  \item $\exists(x\wedge \exists y)=\exists x\wedge \exists y$.
\end{enumerate}}

\emph{(2) From (1), one can see that there exists a one to one correspondence relation between existential and universal quantifiers on NM-algebras. For NM-algebras, we would like to study the monadic filters rather than monadic ideals, using the universal quantifier as the original one is more natural and convenient. Therefore, we introduce the notion of monadic NM-algebras are  NM-algebras equipped with universal quantifiers.}

\emph{(3) The axiomatisation above can be immediately translated into an equational one, so the class of monadic NM-algebras is a variety of algebras. Therefore, the notions of subalgebra and homomorphism are defined as usual.}
\end{remark}

Now, we present some examples for monadic NM-algebras.

\begin{example}\emph{ Let $L$ be an NM-algebra. One can check that the identity $id_L$ is not only a universal quantifier but also is an existential quantifier on $L$. Hence, $(L,id_L)$ is a monadic NM-algebra.}
\end{example}

\begin{example}\emph{Let $L=\{0,a,b,c,d,1\}$, where $0\leq a, b$; $a\leq c,d$; $b\leq c$; $c,d\leq 1$. Define operations $\rightarrow$ and $\odot$ as follows:}
\begin{center}
\begin{tabular}{c|c c c c c c}
   $\rightarrow$ & $0$ & $a$ & $b$ & $c$ & $d$ & $1$\\
   \hline
   $0$ & $1$ & $1$ & $1$ & $1$ & $1$ & $1$ \\
   $a$ & $c$ & $1$ & $c$ & $1$ & $1$ & $1$ \\
   $b$ & $d$ & $d$ & $1$ & $1$ & $d$ & $1$ \\
   $c$ & $a$ & $d$ & $c$ & $1$ & $d$ & $1$\\
   $d$ & $b$ & $c$ & $b$ & $c$ & $1$ & $1$ \\
   $1$ & $0$ & $a$ & $b$ & $c$ & $d$ & $1$
 \end{tabular} {\qquad}
\begin{tabular}{c|c c c c c c}
   $\odot$ & $0$ & $a$ & $b$ & $c$ & $d$ & $1$\\
   \hline
   $0$ & $0$ & $0$ & $0$ & $0$ & $0$ & $0$ \\
   $a$ & $0$ & $0$ & $0$ & $0$ & $0$ & $a$ \\
   $b$ & $0$ & $0$ & $b$ & $b$ & $0$ & $b$ \\
   $c$ & $0$ & $0$ & $b$ & $b$ & $a$ & $c$\\
   $d$ & $0$ & $0$ & $0$ & $0$ & $d$ & $d$ \\
   $1$ & $0$ & $a$ & $b$ & $c$ & $d$ & $1$
 \end{tabular}
\end{center}

\emph{Then  $(\{0,a,b,c,d,1\},\wedge,\vee,\odot,\rightarrow,0,1)$ is an NM-algebra. Now, we define $\forall$ as follows:
$\forall 0=\forall a=0$, $\forall b=\forall c=b$, $\forall d=d$, $\forall 1=1$. One can easily check that $(L,\forall)$ is a monadic NM-algebra.
From Remark 3.3(1), one can see that $\exists 0=0$, $\exists a=\exists d=d$, $\exists b=b$, $\exists c=\exists 1=1$ is an existential quantifier on $L$.
However, $\exists$ is not an existential quantifier on the corresponding residuated lattice $(L,\wedge,\vee,\rightarrow,\odot,0,1)$, since $\exists (a\odot a)=\exists 0=0\neq d=\exists a\odot \exists a$ not hold, which shows that the monadic NM-algebra is not the same as that of monadic bounded residuated lattice in \cite{Rachunek1}.}
\end{example}

\begin{example}\emph{ Let $L$ be a standard NM-algebra on $[0,1]$ and $L_n\subseteq L$ be a standard $n$-valued NM-algebra for some $n\geq 2$ (its elements are $0,\frac{1}{n-1},\cdots,\frac{n-2}{n-1},1)$. For any $x\in L$, we define $\forall x=\max\{y\in L_n|y\leq x\}$ and $\exists x=\min\{y\in L_n|x\leq y\}$. One can easily check that $(L,\forall)$ is a monadic NM-algebra.}
\end{example}

In the following, we will present some useful properties of universal quantifier on a monadic NM-algebra $(L,\forall)$.
\begin{proposition}Let $(L,\forall)$ be a monadic NM-algebra. Then the following properties hold: for any $x, y, z\in L$,
\begin{enumerate}[(1)]
  \item $\forall 0=0$,
  \item $\forall 1=1$,
  \item $\forall\forall x=\forall x$,
  \item $x\leq y$ implies $\forall x\leq \forall y$,
  \item $\forall(x\rightarrow y)\leq \forall x\rightarrow\forall y$, especially, $\forall \neg x\leq \neg \forall x$,
  \item $\forall x\leq y$ if and only if $\forall x\leq \forall y$,
  \item $\forall(\forall x\rightarrow \forall y)=\forall x\rightarrow \forall y$,
  \item $\forall\neg \forall x=\neg \forall x$,
  \item $\forall(x\wedge y)=\forall x\wedge \forall y$,
  \item $\forall(x\odot y)\geq\forall x\odot\forall y$,
  \item $\forall(\forall x\oplus \forall y)=\forall x\oplus \forall y$,
  \item $\forall(x\oplus y)\geq \forall x\oplus \forall y$,
  \item $\forall(\forall x\odot \forall y)=\forall x\odot \forall y$,
  \item $\forall L=L_{\forall}$, where $L_{\forall}=\{x\in L|\forall x=x\}$,
  \item $\forall L$ is a subalgebra of $L$.
\end{enumerate}
\end{proposition}
\begin{proof}\begin{enumerate}[(1)]
               \item Applying (U1), we have $\forall 0 \leq 0$. Thus, $\forall 0=0$.

               \item From (U1) and (U3), we have $\forall 1=\forall(\forall x\rightarrow x)=\forall x\rightarrow \forall x=1$.
               \item From (U4) and (1), we deduce that $\forall\forall x=\forall(0\vee \forall x)=\forall 0\vee \forall x=0\vee \forall x=\forall x$.
               \item If $x\leq y$, then $x\rightarrow y=1$. It follows from (U3),(2) and Proposition 2.2(1) that $1=\forall (1)=\forall(\forall x\rightarrow y)=\forall x\rightarrow \forall y$. Thus, $\forall x\leq \forall y$.
               \item From (U1) and Proposition 2.2(4), we get $x\rightarrow y\leq \forall x\rightarrow y$. Further by (U3) and (4), we have $\forall(x\rightarrow y)\leq \forall(\forall x\rightarrow y)=\forall x\rightarrow \forall y$.
               \item It follows from (U1), (3) and (4).
               \item From (U3) and (3), we deduce that $\forall(\forall x\rightarrow \forall y)=\forall x\rightarrow \forall\forall y=\forall x\rightarrow \forall y$.
               \item From (U3) and (1), we have $\forall\neg \forall x=\forall(\forall x\rightarrow 0)=\forall x\rightarrow \forall 0=\forall x\rightarrow 0=\neg\forall x$.
               \item From (U3) and Proposition 2.2 (9),(10), we have $(\forall x\wedge \forall y)\rightarrow \forall(\forall x\wedge \forall y)=(\forall x\rightarrow \forall(\forall x\wedge \forall y))\vee (\forall y\rightarrow \forall(\forall x\wedge \forall y))=\forall(\forall x\rightarrow (\forall x\wedge \forall y)\vee \forall(\forall y\rightarrow (\forall x\wedge \forall y)=\forall(\forall x\rightarrow \forall y)\vee \forall(\forall y\rightarrow \forall x)=(\forall x\rightarrow \forall y)\vee (\forall y\rightarrow \forall x)=1$. So from Proposition 2.2(1),(U1) and (4), we obtain $\forall x\wedge \forall y\leq \forall(\forall x\wedge \forall y)\leq \forall(x\wedge y)\leq \forall x\wedge \forall y$. Thus, $\forall(x\wedge y)=\forall x\wedge \forall y$.
               \item From $x\odot y\leq x\odot y$, we get $y\leq x\rightarrow(x\odot y)$. Applying (4),(5), we get $\forall y\leq \forall x\rightarrow \forall(x\odot y)$. Thus, $\forall(x)\odot \forall(y)\leq \forall(x\odot y)$ by Definition 2.1(3).
               \item From (U2), we get $\forall(x\oplus \forall y)=\forall x\oplus \forall y$. Further by (3), we have $\forall(\forall x\oplus \forall y)=\forall\forall x\oplus \forall y=\forall x\oplus \forall y$. Thus, $\forall(\forall x\oplus \forall y)=\forall x\oplus \forall y$.
               \item From (U1) and (11), we have $\forall x\oplus \forall y=\forall(\forall x\oplus \forall y)\leq \forall(x\oplus y)$.
               \item It follows from (7) and (8).
               \item Let $y\in \forall L$. So there exists $x\in L$ such that $y=\forall x$. Hence $\forall y=\forall\forall x=\forall x=y$. It follows that $y\in L_{\forall}$. Conversely, if $y\in L_{\forall}$, then we have $y\in\forall L$. Therefore, $\forall L=L_{\forall}$.
               \item (9),(U4) imply that $\wedge$ and $\vee$ are preserved. (7), (13) imply that $\rightarrow$ and $\odot$ are preserved. (1),(2) imply that $0,1\in \forall L$. (3) implies that $\forall$ is preserved. Therefore, $\forall L$ is a subalgebra of $L$.
             \end{enumerate}
\end{proof}

   Now, we turn our attention to some properties of existential quantifier on a monadic NM-algebra $(L,\forall)$.

\begin{proposition}Let $(L,\forall)$ be a monadic NM-algebra. Then the following properties hold: for any $x, y, z\in L$,
\begin{enumerate}[(1)]
  \item $\exists 0=0$,
  \item $\exists 1=1$,
  \item $\exists\exists x=\exists x$,
  \item $x\leq y$ implies $\exists x\leq \exists y$,
  \item $\exists(\exists x\odot \exists y)=\exists x\odot \exists y$,
  \item $\exists\neg \exists x=\neg \exists x$,
  \item $\neg \exists x\leq \exists \neg x$,
  \item $\exists(x\vee y)=\exists x\vee \exists y$,
  \item $x\leq \exists y$ if and only if  $\exists x\leq \exists y$,
  \item $\forall\exists x=\exists x$,
  \item $\exists\forall x=\forall x$,

  \item $\forall x=x$ if and only if $\exists x=x$,
  \item $\exists L=L_{\exists}$, where $L_{\exists}=\{x\in L|\exists x=x\}$,
  \item $\exists L=\forall L$,
  \item $(\exists,\forall)$ establishes a Galois connection over $(L,\leq)$,
  \item $\forall(\exists x\rightarrow \exists y)=\exists x\rightarrow \exists y$,
   \item $\exists(\exists x\oplus\exists y)=\exists x\oplus \exists y$.
\end{enumerate}
\end{proposition}
\begin{proof} The proofs are dual to that of Proposition 3.7, so we omit them.
\end{proof}

In the following, we will show that the converses of Proposition 3.7 (12), Proposition 3.8 (7) are not true, in general.

\begin{example}\emph{(1) Consider the monadic NM-algebra $(L,\forall)$ in Example 3.5, we have $\forall(c\oplus c)=\forall 1=1\nleq b=\forall c\oplus \forall c$. Therefore, the converse of Proposition 3.7 (12) is not true, in general.}

\emph{(2) Let $L=[0,1]$. For any $x,y\in L$, we define $x\wedge y=\min\{x,y\}$, $x\vee y=\max\{x,y\}$, $x\odot y=0$ if $x\leq y\rightarrow 0$; otherwise $x\odot y=x\wedge y$ and $x\rightarrow y=1$ if $x\leq y$; otherwise $x\rightarrow y=\neg x\vee y$. Then $([0,1],\wedge,\vee,\rightarrow,\odot,0,1)$ is an NM-algebra. Now we define $\forall$ and $\exists$ as follows:}
\begin{center}
$\forall x=
\begin{cases}
1, & x=1 \\
0, & x\neq 1
\end{cases}$
{\qquad}
$\exists x=
\begin{cases}
0, & x=0 \\
1, & x\neq 0
\end{cases}$
\end{center}

\emph{One can easily check that $(L,\forall)$ is a monadic NM-algebra. Moreover, we have $\exists \neg\frac{1}{2}=1\nleq 0=\neg \exists \frac{1}{2}$. Therefore, the converse of Proposition 3.8(7) is not true, in general.}
\end{example}

In the following, we show that monadic NM-algebras faithfully the axioms on quantifiers in monadic predicate NM logic, which was introduced in \cite{Bianchi}.

\begin{theorem} Let $L$ be an NM-algebra, $\forall:L\longrightarrow L$ and $\exists:L\longrightarrow L$ be two mappings. Then the sets of $G=\{U1,U2,U3,U4\}$ and $H=\{W1,W2,W3,W4,W5\}$ are equivalent, where $W1-W5$ defined as follows: for any $x,y\in L$,
\begin{enumerate}[(W1)]
  \item $\forall x\rightarrow x=1$,
  \item $x\rightarrow\exists x=1$,
  \item $\forall(x\rightarrow \exists y)=\exists x\rightarrow \exists y$,
  \item $\forall(\exists x\rightarrow y)=\exists x\rightarrow \forall y$,
  \item $\forall(x\vee \exists y)=\forall x\vee \exists y$.
\end{enumerate}
\end{theorem}
\begin{proof} $G\Rightarrow H$: From Remark 3.3(1), we have $\exists x=\neg\forall\neg x$.
                              \begin{enumerate}[(W1)]
                                  \item It follows from (U1).
                                  \item Applying (U1), we have $\forall \neg x\leq \neg x$ and hence $\neg\forall \neg x\geq x$. Therefore, $x\leq \exists x$.
                                  \item From (U2), we have $\forall(x\rightarrow \forall y)=\forall(\neg\neg x\rightarrow \forall y)=\neg\forall \neg x\rightarrow \forall y=\exists x\rightarrow \forall y$.
                                  \item From Proposition 3.8(10), we have $\forall\exists x=\exists x$, for any $x\in L$. Moreover, by (U3), one can obtain that $\forall(\exists x\rightarrow y)=\forall(\forall\exists x\rightarrow y)=\forall\exists x\rightarrow y=\exists x\rightarrow\forall y$.
                                  \item From Proposition 3.8(10), we have $\forall\exists x=\exists x$, for any $x\in L$. Moreover, by (U4), one can obtain that $\forall( x\vee \exists y)=\forall(x\vee \forall\exists y)=\forall x\vee \forall\exists y=\forall x\vee\exists y$.
                                \end{enumerate}
             $H\Rightarrow G$: From $W1,W2,W3,W4,W5$, we have $\neg \forall x=\exists \neg x$, $\forall\exists x=\exists x$ and $\exists\forall x=\forall x$.
             \begin{enumerate}[(U1)]
               \item It follows from (W1).
               \item From (W3), we have $\forall (\neg x\rightarrow \forall y)=\forall(\neg x\rightarrow \exists\forall y)=\exists \neg x\rightarrow \exists\forall y=\neg\forall x\rightarrow \forall y$.
               \item From (W4), we have $\forall(\forall x\rightarrow y)=\forall(\exists\forall x\rightarrow y)=\exists\forall x\rightarrow \forall y=\forall x\rightarrow \forall y$.
               \item From (W5), we have $\forall(x\vee \forall y)=\forall(x\vee \exists\forall y)=\forall x\vee \exists\forall y=\forall x\vee \forall y$.
             \end{enumerate}
\end{proof}

In the following, we will give some conditions under which
monadic NM-algebras become monadic Boolean algebras, which were
introduced by Halmos as a pair $(L,\exists)$ satisfying conditions
(E1),(E4) and Proposition 3.8(1).

\begin{theorem} Let $(L,\forall)$ be a monadic NM-algebra. Then the following statements are equivalent:
\begin{enumerate}[(1)]
  \item $(L,\exists)$ is a monadic Boolean algebra,
  \item every universal quantifier $\forall$ on $L$ satisfies $\forall(x\wedge y)=\forall x\odot \forall y$ for any $x,y\in L$,
  \item every universal quantifier $\forall$ on $L$ satisfies $\forall(x\vee y)=\forall x\oplus \forall y$ for any $x,y\in L$.
\end{enumerate}
\end{theorem}
\begin{proof}$(1)\Rightarrow(2)$ Suppose that $L$ is a Boolean algebra and $\forall$ is any quantifier on $L$. Then $L$ satisfies the property  $x\odot y = x\wedge y$ for any $x,y\in L$. Applying Proposition 3.7(9), we have $\forall(x\wedge y)=\forall x\wedge \forall y=\forall x\odot \forall y$. Therefore, we have $\forall(x\wedge y)=\forall x\odot \forall y$.

$(2)\Rightarrow(1)$ Assume that  every universal quantifier $\forall$ on $L$ satisfies $\forall(x\wedge y)=\forall x\odot \forall y$ for any $x,y\in L$. From Example 3.4, we know that $id_L$ is an universal quantifier on $L$. Thus, taking $\forall=id_L$, we have $x\odot y=x\wedge y$ for all $x,y\in L$.  Therefore, $L$ is a Boolean algebra and hence $(L,\exists)$ is a monadic Boolean algebra.

$(1)\Rightarrow(3)$  Suppose that $L$ is a Boolean algebra
and $\forall$ is any quantifier on $L$.
Then $L$ satisfies the property $x\oplus y = x\vee y$ for any $x,y\in L$. From Proposition 3.7(12), we have $\forall(x\vee y) = \forall(x\oplus y)\leq \forall x \oplus \forall y$. Further by Proposition 3.7(4), we have $\forall(x\vee y)\geq \forall x\vee \forall y=\forall x\oplus \forall y$. Thus, $\forall(x\vee y)= \forall x\oplus \forall y$.

$(3)\Rightarrow(1)$ Assume that  every universal quantifier $\forall$ on $L$ satisfies $\forall(x\vee y)=\forall x\oplus \forall y$ for any $x,y\in L$. Taking $\forall=id_L$, we have $x\vee y=x\oplus y$ for all $x,y\in L$. Therefore, $L$ is a Boolean algebra and hence $(L,\exists)$ is a monadic Boolean algebra.
 \end{proof}

It is interesting to note that for linearly ordered monadic NM-algebra, we have $\forall(x\vee y)=\forall x\vee\forall y$ . So for these subvarieties the axiom (U4) can be rewritten in the form:
\begin{center}
 (U4')  $\forall(x\vee y)=\forall x\vee \forall y$.
 \end{center}

  Motivated by the above consideration, we introduce a special kind of universal quantifier under the name of strong universal quantifier on an NM-algebra.

\begin{definition}\emph{A \emph{strong universal quantifier} on an NM-algebra $L$ is a mapping $\forall:L\rightarrow L$ satisfying (U1),(U2),(U3) and (U4'). The couple $(L,\forall)$ is called a \emph{strong monadic NM-algebra}.}
\end{definition}

\begin{example}\emph{ Consider the monadic NM-algebra $(L,\forall)$ in Example 3.9(2), one can check that it is also a strong monadic NM-algebra. However, the strong universal quantifier $\forall$ is not a homomorphism on $L$, since $\forall(\frac{1}{2}\rightarrow 0)=0\neq 1= \forall \frac{1}{2}\rightarrow \forall 0$.}
\end{example}

\begin{proposition} Every strong monadic NM-algebra is a monadic NM-algebra.
\end{proposition}
\begin{proof} Let $\forall$ be a strong universal quantifier on $L$. We prove that (U4') implies (U4). Indeed, we have $\forall(x\vee \forall y)=\forall x\vee \forall\forall y=\forall x\vee \forall y$, which is axiom (U4). Thus, we obtain that (U4') implies (U4).
\end{proof}

However, the converse of Proposition 3.14 is not true, in general.

\begin{example}\emph{ Consider the monadic NM-algebra $(L,\forall)$ in Example 3.5, one can easily check that it is a monadic NM-algebra but not a strong monadic NM-algebra, since $\forall(a\vee b)=\forall d=d\neq b=\forall a\vee \forall b$. }
\end{example}

In what follows, we will discuss relations between monadic NM-algebras and some related structures, likeness modal NM-algebras and rough approximation spaces.

\begin{definition}\emph{\cite{Duan} A \emph{modal NM-algebra} is a structure $(L,\wedge,\vee,\odot,\rightarrow,0,1,\tau)$, where $(L,\wedge,\vee,\odot,\rightarrow,0,1)$ is an NM-algebra and $\tau:L\rightarrow L$ is a unary operator on $L$, called a modal operator, such that the following conditions are satisfied:
\begin{enumerate}[(M1)]
  \item $\tau(1)=1$,
  \item $\tau(x\vee y)\leq \tau (x)\vee \tau (y)$,
  \item $\tau(x\rightarrow y)\leq\tau(x)\rightarrow \tau(y)$,
  \item $\tau(x)\leq \tau\tau(x)$,
  \item $\tau(x)\leq x$.
\end{enumerate}
 for any $x,y\in L$.}
\end{definition}

In what follows, we will show that a modal NM-algebra is a strong monadic NM-algebra if it satisfies the condition $(\ast)$: $\forall(\forall x\rightarrow \forall y)=\forall x\rightarrow \forall y$ for any $x,y\in L$.

\begin{theorem}Let $L$ be an NM-algebra and $\forall:L\rightarrow L$ be a unary operator on $L$. Then the following statements are equivalent:
\begin{enumerate}[(1)]
  \item $(L,\forall)$ is a strong monadic NM-algebra,
  \item $(L,\forall)$ is a modal NM-algebra with the condition $(\ast)$.
\end{enumerate}
\end{theorem}
\begin{proof} $(1)\Rightarrow(2)$ Conditions (M1),(M2),(M3),(M4)and (M5) are directly contained in the definition and proposition of a strong monadic NM-algebra as (U1),(U4') and Proposition 3.7(2),(3),(5).

$(2)\Rightarrow(1)$ Conditions (U1),(U4') are directly contained in the definition of a modal NM-algebra. In order to show (U3), from (M3),(M4) and (M5) we have $\forall(\forall x\rightarrow y)\leq \forall\forall x\rightarrow\forall y=\forall x\rightarrow \forall y$. On the other hand, from (M5) and Proposition 2.2(4) we obtain that $\forall x\rightarrow \forall y\leq \forall x\rightarrow y$, where using $\forall(\forall x\rightarrow \forall y)=\forall x\rightarrow \forall y$ we get $\forall x\rightarrow \forall y\leq \forall(\forall x\rightarrow y)$. Finally, the proof of (U2) as following: from Proposition 3.7(8) and (U3), we get $\forall(\neg x\rightarrow \forall y)=\forall(x\oplus \forall y)=\forall(\forall y\oplus x)=\forall(\neg\forall y \rightarrow x)=\forall(\forall\neg \forall y\rightarrow x)=\forall\neg \forall y\rightarrow \forall x=\neg\forall y\rightarrow \forall x=\forall x\oplus \forall y=\neg \forall x\rightarrow \forall y$.
\end{proof}

\begin{definition}\emph{\cite{She} A rough approximation space is a system $\mathcal{R}=(X,L(X),(U(X))$, where}
\begin{enumerate}[(1)]
  \item\emph{ $(X,\leq,0,1)$ is a poset with respect to the order $\leq$, elements from $X$ are said to be approximable elements.
  \item $L(X)$ is a subpoest of $X$ containing of all available inner definable elements.
   \item $U(X)$ is a subpoest of $X$ containing of all available upper definable elements,}\\
\emph{  and satisfying the following axioms:}

      \begin{enumerate}[(a)]
        \item\emph{ for any approximable element $x\in X$, there exists
one element $i(x)$ if it satisfies: (1) $i(x)\in L(X)$; (2) $i(x)\leq x$; (3) for any $a\in L(X)$, $(a\leq x\Rightarrow a\leq i(x))$.
We called $i$ an inner approximation map from $X$ to $L(X)$.
      \item For any approximable element $x\in X$, there exists one element
$u(x)$ if it satisfies: (1) $u(x)\in U(X)$; (2) $x\leq U(x)$; (3)
for any $r\in U(X)$, $(x\leq r\Rightarrow u(x)\leq r)$. We called $i$ a upper approximation map from $X$ to $U(X)$.}
\end{enumerate}

\end{enumerate}
\end{definition}

\begin{theorem} A monadic NM-algebra $(L,\forall)$ induce a rough approximation space $\mathcal{R}=(L,\forall L,\exists L)$ in which,
\begin{enumerate}[(1)]
  \item $L$ is the set of approximable elements,
  \item $\forall L$ is the set of exact or definable elements,
  \item $\exists:L\rightarrow \exists L$ is the upper approximation map, satisfying (for any $x\in \exists L$)(for any $y\in L)(x\leq y$ iff $\exists x\leq y)$,
  \item $\forall:L\rightarrow \forall L$ is the inner approximation map, satisfying (for any $x\in \forall L$) (for any $y\in L)(x\leq y$ iff $x\leq \forall y)$, in which for any element $x$ in $L$, its rough approximation is defined by $(\forall x, \exists x)$.
\end{enumerate}
\end{theorem}

\begin{proof} Suppose that $x\in \forall L$ and $y\in L$. If $x\leq y$, then $x=\forall x\leq \forall y$. Conversely, if $x\leq \forall y$, then $x\leq \forall y\leq y$.

Suppose that $x\in L$ and $y\in \exists L$. If $x\leq y$, then $\exists x\leq \exists y=y$. Conversely, if $\exists x\leq y$, then $x\leq \exists x\leq y$.
\end{proof}

{\bf Open Problem:} Whether there exists a nontrival universal quantifier $\forall$, ensures that no two different elements have the same rough approximation?

\section{Monadic filters in monadic NM-algebras}
In this section, we introduce and investigate monadic filters in
monadic NM-algebras. In particular, we prove a representation
theorem for monadic NM-algebras and characterize two kinds of
monadic NM-algebras, which are simple and subdirectly irreducible
monadic NM-algebras.

\begin{definition}\emph{ Let $(L,\forall)$ be a monadic NM-algebra. A nonempty subset $F$ of $L$ is called a \emph{monadic
filter} of $(L,\forall)$, if $F$ is a filter of $L$ such that if $x\in F$, then $\forall x\in F$ for all $x\in L$.}
\end{definition}

We will denote the set of all monadic filters of $(L,\forall)$ by $MF[L]$.
\begin{example} \emph{Consider the monadic NM-algebra
 $(L,\forall)$ in Example 3.5, one can easily check
that the monadic filters of
$(L,\forall)$ are $\{1\}$,$\{1,d\}$, $\{1,b,c\}$ and $L$.
However, consider the monadic NM-algebra $(L,\forall)$ in Example 3.9(2),
 one can check that  $(\frac{1}{2},1]$ is a filter of $L$ but it is not
a monadic filter of $(L,\forall)$.}
\end{example}

Let $(L,\forall)$ be a monadic NM-algebra. For any nonempty subset $X$ of $L$, we denote by $\langle X \rangle_\forall$ the monadic filter of $(L,\forall)$ generated by $X$, that is, $\langle X \rangle_\forall$ is the smallest monadic filter of $(L,\forall)$ containing $X$.
If $F$ is a monadic filter of $(L,\forall)$ and $x\notin F$, then we put $\langle F,x\rangle_\forall:=\langle F\cup \{x\}\rangle_\forall$. \\

The next theorem gives a concrete description of the monadic filter generated by a nonempty subset $X$ of an NM-algebra $L$.
\begin{theorem} Let $(L,\forall)$ be a monadic NM-algebra and $X$ be a nonempty set of $L$. Then \\
$\langle X\rangle_\forall=\{x\in L|x\geq \forall x_1\odot\forall x_2\odot\cdots \odot\forall x_n, x_i\in X, n\geq 1\}$.
\end{theorem}
\begin{proof} The proof is easy, and hence we omit the details.
\end{proof}
\begin{theorem} Let $F$, $F_1$, $F_2$ be monadic filters of  $(L,\forall)$ and $a\notin F$. Then
\begin{enumerate}[(1)]
  \item $\langle a\rangle_\forall=\{x\in L|x\geq (\forall a)^n, n\geq 1\}$,
  \item $\langle F\cup a\rangle_\forall=\{x\in L|x\geq f\odot (\forall a)^n, f\in F\}=F\vee [\forall a)$,
  \item $\langle F_1\cup F_2\rangle_\forall=\{x\in L|x\geq f_1\odot f_2, f_1\in F_1, f_2\in F_2\}=F_1\vee F_2$,
  \item if $a\leq b$, then $\langle b\rangle_\forall\subseteq \langle a\rangle_\forall$,
  \item $\langle \forall a\rangle_\forall=\langle a\rangle_\forall$,
  \item $\langle a\rangle_\forall\vee \langle b\rangle_\forall=\langle a\wedge b\rangle_\forall=\langle a\odot b\rangle_\forall$,
  \item if $(L,\forall)$ is a strong monadic NM-algebra, then $\langle a\rangle_\forall\cap\langle b\rangle_\forall=\langle \forall a\vee \forall b\rangle_\forall$.
\end{enumerate}
\end{theorem}
\begin{proof} The proofs of $(1)-(5)$ are obvious.

$(6)$ Since $a\odot b\leq a\wedge b\leq a,b$, we deduce that $\langle a\rangle_\forall$, $\langle b\rangle_\forall\subseteq\langle a\wedge b\rangle_\forall\subseteq \langle a\odot b\rangle_\forall$. It follows from that $\langle a\rangle_\forall \vee \langle b\rangle_\forall\subseteq \langle a\wedge b\rangle_\forall \subseteq \langle a\odot b\rangle_\forall$. Conversely, let $a\in \langle a\odot b\rangle_\forall$. Then, for some natural number $n\geq 1$, $a\geq (\forall(a\odot b))^n\geq (\forall a\odot \forall b)^n=(\forall a)^n\odot (\forall b)^n$. Hence $a\in \langle a\rangle_\forall \vee \langle b\rangle_\forall$, we deduce that $\langle a\odot b\rangle_\forall \subseteq \langle a\rangle_\forall \vee \langle b\rangle_\forall$. Therefore, $\langle a\rangle_\forall \vee \langle b\rangle_\forall=\langle a\wedge b\rangle_\forall=\langle a\odot b\rangle_\forall$.

(7) Since $\forall a\leq \forall a\vee \forall b$, we deduce that $\langle\forall a\vee \forall b\rangle_\forall \subseteq \langle \forall a\rangle_\forall=\langle a\rangle_\forall$. Analogously, $\langle\forall a\vee \forall b\rangle_\forall \subseteq \langle \forall b\rangle_\forall=\langle b\rangle_\forall$. It follows that $\langle\forall a\vee \forall b\rangle_\forall \subseteq \langle a\rangle_\forall \cap \langle b\rangle_\forall$. Moreover, let $t\in \langle a\rangle_\forall \cap \langle b\rangle_\forall$. Then for some natural number $n,m\geq 1$, $t\geq (\forall a)^m$ and $t\geq (\forall b)^n$. Hence $t\geq (\forall a)^m\vee (\forall b)^n\geq (\forall a\vee\forall b)^{mn}=(\forall(a\vee b))^{mn}$, we deduce that $a\in \langle \forall a\vee \forall b\rangle_\forall$, that is, $\langle a\rangle_\forall\cap\langle b\rangle_\forall\subseteq\langle \forall a\vee \forall b\rangle_\forall$. Therefore, $\langle a\rangle_\forall\cap\langle b\rangle_\forall=\langle \forall a\vee \forall b\rangle_\forall$.
\end{proof}
\begin{theorem}Let $(L,\forall)$ be a monadic NM-algebra and $F$ be a filter of $L$. Then the following statements are equivalent:
\begin{enumerate}[(1)]
  \item $F$ is a monadic filter of $(L,\forall)$,
  \item $F=\langle F\cap L_{\forall}\rangle$.
\end{enumerate}
\end{theorem}
\begin{proof}  It is similar to the proof of Lemma 9 in \cite{Bezhanishvili}.
\end{proof}



\begin{corollary} Let $(L,\forall)$ be a monadic NM-algebra. Then the lattice $MF[L]$  is isomorphic to the lattice $F[L_{\forall}]$ of all filters of the NM-algebra $L_{\forall}$.
\end{corollary}
\begin{proof} It is similar to the proof of Corollary 10 in \cite{Bezhanishvili}.

\end{proof}

\begin{proposition} Let $(L,\forall)$ be a monadic NM-algebra, $x\in L_{\forall}$ and $F$ be a monadic filter of $(L,\forall)$. Then $\langle F\cup\{x\}\rangle$ is a monadic filter of $(L,\forall)$.
\end{proposition}
\begin{proof} It is similar to the proof of Lemma 8 in \cite{Bezhanishvili}.
\end{proof}

\begin{definition}\emph{Let $(L,\forall)$ be a monadic NM-algebra and $\theta$ be a congruence on  $L$. Then $\theta$ is called a \emph{monadic congruence} on $(L,\forall)$ if $(x,y)\in\theta$ implies $(\forall x,\forall y)\in \theta$, for any $x,y\in L$.}
\end{definition}

We will denote the set of all monadic congruences of $(L,\forall)$ by $MC[L]$.
\begin{theorem}For any monadic NM-algebra, there exists a one to one correspondence between its monadic filters and its monadic congruences.
\end{theorem}
\begin{proof}  It is similar to the proof of Theorem 11 in \cite{Bezhanishvili}.
 \end{proof}



Let $(L,\forall)$ be a monadic NM-algebra and $F$ be a monadic filter. We define a mapping $\forall_F: L/F\rightarrow L/F$ such that
 $\forall_F([x])=[\forall x]$, for any $x\in L$.
\begin{proposition}Let $(L,\forall)$ be a monadic NM-algebra and $F$ a monadic filter of $(L,\forall)$. Then $(L/F,\forall_F)$ is a monadic NM-algebra.
\end{proposition}
\begin{proof} From the proof of Theorem 4.9, one can see that the mapping $\forall_F$ is well defined. Moreover, the (U1)-(U4) in Definition 3.1 are checked easily, so we omit them.
\end{proof}
\begin{definition}\emph{Let $(L,\forall)$ be a monadic NM-algebra. A proper monadic filter $F$ of $(L,\forall)$ is called a \emph{prime monadic filter} of $(L,\forall)$, if for all monadic filter $F_1$, $F_2$ of $(L,\forall)$ such that $F_1\cap F_2\subseteq F$, then $F_1\subseteq F$ or $F_2\subseteq F$.}
\end{definition}
\begin{example}\emph{Consider the Example 3.5, one can easily obtain that $\{d,1\}$ and $\{b,c,1\}$ are prime monadic filters of $(L,\forall)$. Moreover, one can check that $\{1\}$ is a monadic filter of $(L,\forall)$, but it is not a prime monadic filter. In fact, $\{d,1\}$ and $\{b,c,1\}$ are monadic filters of $(L,\forall)$ and $\{1\}\subseteq \{d,1\}\cap \{b,c,1\}$, but $\{d,1\}\nsubseteq \{1\}$ and $\{b,c,1\}\nsubseteq \{1\}$.}
\end{example}
\begin{theorem}Let $(L,\forall)$ be a strong monadic NM-algebra and $F$ be a proper monadic filter of $(L,\forall)$. Then the following statements are equivalent:
\begin{enumerate}[(1)]
  \item $F$ is a prime monadic filter of $(L,\forall)$,
  \item if $\forall x\vee \forall y\in F$ for some $x,y\in L$, then $\forall x\in F$ or $\forall y\in F$,
  \item $\forall x\rightarrow \forall y\in F$ or $\forall y\rightarrow\forall x\in F$ for any $x,y\in L$,
  \item $(L/F,\forall_F)$ is a linearly ordered monadic NM-algebra.
\end{enumerate}
\end{theorem}
\begin{proof}$(1)\Rightarrow(2)$ Let $\forall x\vee \forall y\in F$ for some $x,y\in L$. Then $\langle x\rangle_\forall\cap\langle y\rangle_\forall=\langle \forall x\vee \forall y\rangle_\forall\subseteq F$. If $F$ is a prime monadic filter of $(L,\forall)$, then, we have $\langle x\rangle_\forall\subseteq F$ or $\langle y\rangle_\forall\subseteq F$. Therefore, $\forall x\in F$ or $\forall y\in F$.

$(2)\Rightarrow (1)$ Suppose that $F_1,F_2\in MF[L]$ such that $F_1\cap F_2\subseteq F$, $F_1\nsubseteq F$ and $F_2\nsubseteq F$. Then, there exist $x\in F_1$ and $y\in F_2$ such that $x,y\notin F$. Since $F_1,F_2$ is are monadic filters of $(L,\forall)$, then $\forall x\in F_1$ and $\forall y\in F_2$. From $\forall x$, $\forall y\leq \forall x\vee \forall y$, we obtain that $\forall x\vee \forall y\in F_1\cap F_2\subseteq F$. Further by (2), we get $x\in F$ or $y\in F$, which is a contradiction. Therefore, $F$ is a prime monadic filter of $(L,\forall)$.

$(2)\Leftrightarrow (3)$ It is similar to the proof of Proposition 8.5.3 in \cite{Wang1}.

$(3)\Leftrightarrow(4)$  It is similar to the proof of Proposition 8.5.6 in \cite{Wang1}.
\end{proof}

\begin{definition}\emph{Let $(L,\forall)$ be a monadic NM-algebra. A proper monadic filter $F$ of $(L,\forall)$ is called a \emph{maximal monadic filter} if it not strictly contained in any proper monadic filter of $(L,\forall)$.}
\end{definition}
\begin{example}\emph{Let $L=\{0,a,b,c,d,e,f,g,1\}$, where $0\leq a, b\leq c\leq d\leq e\leq f,g\leq 1$. Define operations $\rightarrow$ and $\odot$ as follows:}
\begin{center}
\begin{tabular}{c|c c c c c c c c c}
   $\rightarrow$ & $0$ & $a$ & $b$ & $c$ & $d$ & $e$ & $f$ & $g$ & $1$\\
   \hline
   $0$ & $1$ & $1$ & $1$ & $1$ & $1$ & $1$ & $1$ & $1$ & $1$ \\
   $a$ & $g$ & $1$ & $f$ & $1$ & $1$ & $1$ & $1$ & $1$ & $1$ \\
   $b$ & $f$ & $g$ & $1$ & $1$ & $1$ & $1$ & $1$ & $1$ & $1$ \\
   $c$ & $e$ & $g$ & $f$ & $1$ & $1$ & $1$ & $1$ & $1$ & $1$\\
   $d$ & $d$ & $d$ & $d$ & $d$ & $1$ & $1$ & $1$ & $1$ & $1$ \\
   $e$ & $c$ & $c$ & $c$ & $c$ & $d$ & $1$ & $1$ & $1$ & $1$ \\
   $f$ & $b$ & $a$ & $c$ & $c$ & $d$ & $g$ & $1$ & $g$ & $1$ \\
   $g$ & $a$ & $c$ & $b$ & $c$ & $d$ & $f$ & $f$ & $1$ & $1$\\
   $1$ & $0$ & $a$ & $b$ & $c$ & $d$ & $e$ & $f$ & $g$ & $1$
 \end{tabular} {\qquad}
\begin{tabular}{c|c c c c c c c c c}
   $\odot$ & $0$ & $a$ & $b$ & $c$ & $d$ & $e$ & $f$ & $g$ & $1$\\
   \hline
   $0$ & $0$ & $0$ & $0$ & $0$ & $0$ & $0$ & $0$ & $0$ & $0$ \\
   $a$ & $0$ & $0$ & $0$ & $0$ & $0$ & $0$ & $b$ & $0$ & $0$ \\
   $b$ & $0$ & $0$ & $0$ & $0$ & $0$ & $0$ & $0$ & $a$ & $b$ \\
   $c$ & $0$ & $0$ & $0$ & $0$ & $0$ & $0$ & $a$ & $b$ & $c$\\
   $d$ & $0$ & $0$ & $0$ & $0$ & $0$ & $d$ & $d$ & $d$ & $d$ \\
   $e$ &  $0$ & $0$ & $0$ & $0$ & $d$ & $e$ & $e$ & $e$ & $e$ \\
   $f$ & $0$ & $a$ & $0$ & $a$ & $d$ & $e$ & $e$ & $g$ & $f$ \\
   $g$ & $0$ & $0$ & $b$ & $b$ & $d$ & $e$ & $f$ & $e$ & $g$\\
   $1$ & $0$ & $a$ & $b$ & $c$ & $d$ & $e$ & $f$ & $g$ & $1$
 \end{tabular}
\end{center}

\emph{Then  $(\{0,a,b,c,d,e,f,g,1\},\wedge,\vee,\odot,\rightarrow,0,1)$ is an NM-algebra. Now, we define $\forall$ as follows:
$\forall 0=\forall a=\forall b=0$, $\forall c=c$, $\forall d=d$, $\forall e=\forall f=\forall g=e$, $\forall 1=1$. One can easily check that $(L,\forall)$ is a monadic NM-algebra and  $\{e,f,g,1\}$ is a maximal monadic filter of $(L,\forall)$. Moreover, one can check that $\{e,g,1\}$ and $\{e,f,1\}$ are monadic filters of $(L,\forall)$ but not maximal monadic filters of $(L,\forall)$.}

\end{example}
\begin{theorem} Let $F$ be a proper monadic filter of $(L,\forall)$. Then the the following statements are equivalent:
\begin{enumerate}[(1)]
  \item $F$ is a maximal monadic filter of $(L,\forall)$,
  \item $\forall x\in F$ or $\neg \forall x\in F$ for any $x\in L$
  \item $\exists x\in F$ or $\neg \exists x\in F$ for any $x\in L$,.
\end{enumerate}
\end{theorem}
\begin{proof} It is similar to the proof of Theorem 16 in \cite {Bezhanishvili}.


\end{proof}

\begin{theorem} Let $(L,\forall)$ be a strong monadic NM-algebra, $F$ be a  monadic filter of $(L,\forall)$ and $a\notin F$. Then there exists a prime monadic filter $P$ of $(L,\forall)$ such that $F\subseteq P$ and $a\notin P$.
\end{theorem}
\begin{proof}Denote $F_a=\{F^{\prime}|F^{\prime}$ is a proper monadic filter of $(L,\forall)$ such that $F\subseteq F^{\prime}$, $a\notin F^{\prime}\}$. Then $F_a\neq\emptyset$ since $F$ is a monadic filter not containing $a$ and $F_a$ is a partially set under inclusion relation. Suppose that $\{F_i|i\in I\}$ is a chain in $F_a$, then $\cup\{F_i|i\in I\}$ is a monadic filter of $(L,\forall)$ and it is the upper bounded of this chain. By Zorn's Lemma, there exists a maximal element $P$ in $F_a$. Now, we shall prove that $P$ is desired prime monadic filter of ours. Since $P\subseteq F_a$, then $P$ is a proper monadic filter and $a\notin P$.

Let $ x\vee y\in P$ for some $x,y\in L$. Suppose that $x\notin P$ and $y\notin P$. Since $P$ is strictly contained in $\langle P,x\rangle_\forall$ and $\langle P,y\rangle_\forall$ and by the maximality of $P$, we deduce that $\langle P,x\rangle_\forall\notin F_a$ and $\langle P,y\rangle_\forall\notin F_a$. Then $a\in \langle P,x\rangle_\forall=P\vee [\forall x)$ and $a\in \langle P,y\rangle_\forall=P\vee [\forall y)$. It follows from strong property of monadic NM-algebra, we have $a\in (P\vee [\forall x))\wedge (P\vee [\forall y))=P\vee ([\forall x)\wedge [\forall y))=P\vee [\forall x\vee \forall y)=P\vee [\forall(x\vee y))\in P$, which implies $a\in P$, a contradiction. Therefore, $P$ is a prime monadic filter such that $F\subseteq P$ and $a\notin P$.
\end{proof}

  The following theorem gives a representation theorem of monadic NM-algebras.
\begin{theorem}Each strong monadic NM-algebra is a subalgebra of the direct product of a system of linearly ordered monadic NM-algebras.
\end{theorem}
\begin{proof} The proof of this theorem is as usual and the only critical point is the above Theorem 4.17.
\end{proof}

The following theorem gives a characterization of representable NM-algebras.

\begin{theorem} Let $(L,\forall)$ be an monadic NM-algebra. Then the following statements are equivalent:
\begin{enumerate}[(1)]
  \item $L$ is representable,
  \item for any strong universal quantifier $\forall$, $(L,\forall)$ is representable.
\end{enumerate}
\end{theorem}
\begin{proof}$(1)\Rightarrow(2)$ Suppose that the NM-algebra $L$ is representable. Then by Theorem 2.7, there exists a system $S$ of prime filter of $L$ such that $\bigcap S=\{1\}$. Since every prime filter of $L$ contains a minimal prime filter, we get that in our case the intersection of all minimal prime filter is equal to $\{1\}$. Moreover, we will show that every minimal prime filter of $L$ is a prime monadic filter in $(L,\forall)$. Let $P$ be a minimal prime filter of $L$. Then by Theorem 2.5, $P=\cup\{a^\bot|a\in P\}$. If $x\in P$, then there is $a\notin P$ such that $x\vee a=1$, hence $1=\forall 1=\forall(x\vee a)=\forall x\vee \forall a$. Since $a\notin P$, we get $\forall a\notin P$, and hence  $\forall x\in P$, that means that $P$ is a prime monadic filter in $(L,\forall)$. Applying Theorem 2.7 again, one can see that $(L,\forall)$ is a subdirect product of linearly ordered monadic NM-algebras.

$(2)\Rightarrow(1)$ Assume that any strong universal quantifier $\forall$ such that $(L,\forall)$ is representable. From Example 3.4, we know that $id_L$ is a strong universal quantifier on $L$. Thus, taking $\forall=id_L$, we can obatin that $L$ is representable.
\end{proof}

Now, we introduce two kinds of monadic NM-algebras and give some characterizations of them.
\begin{definition}\emph{ A monadic NM-algebra $(L,\forall)$ is said to be \emph{simple} if it has exactly two monadic filters: $\{1\}$ and $L$.}
\end{definition}
\begin{example}\emph{ Consider the monadic NM-algebra $(L,\forall)$ in Example 3.9(2)
, one can easily check that it is a simple monadic NM-algebra.}
\end{example}

The following theorem gives some characterizations of simple monadic NM-algebras.
\begin{theorem} $(L,\forall)$ be a monadic NM-algebra. Then the following statements are equivalent:
\begin{enumerate}[(1)]
  \item $(L,\forall)$ is simple,
  \item $\forall L$ is simple,
  \item $L_{\forall}=\{0,1\}$,
  \item $\langle 1\rangle$ is the only proper monadic filter in $(L,\forall)$.
\end{enumerate}
\end{theorem}
\begin{proof}$(1)\Rightarrow(2)$ If $(L,\forall)$ is simple, $F$ is a filter of $\forall L$ and $F\neq\{1\}$.  Now, we will prove that $\forall L$ is simple. Consider the set $F_f=\{z\in L|z\geq f$ for a certain $f\in F\}$. If $x,y\in F_f$, then there exist $f_1, f_2\in F$ such that $x\geq f_1, y\geq f_2$, so $x\odot y\geq f_1\odot f_2\in F$, and thus $x\odot y\in F_f$. Moreover, if $x\in F_f$ and $x\leq y$, then $y\in F_f$. Furthermore, if $x\in F_f$, then $x\geq f, f\in F$, and hence $\forall x\geq \forall f=f$ (since $f\in \forall L$), that is, $\forall x\in F_f$. Therefore, $F_f$ is a monadic filter of $(L,\forall)$. Since $(L,\forall)$ is simple, and $F_f\neq\{1\}$ (since $F\subseteq F_f$). It follows that $F_f=L$, and so $0\in F_f$, hence $F=\forall L$. From Definition 2.4(1), we obtain that $\forall L$ is simple.

$(2)\Rightarrow(1)$ Let $F$ be a monadic filter of $(L,\forall)$. Then $F\cap \forall L$ is a filter of $\forall L$, and so $F\cap \forall L=\{1\}$ or $F\cap \forall L=\forall L$. If $F\cap \forall L=\forall L$, then $\forall L\subseteq F$. Since $0\in\forall L$, we deduce $F=L$. If $F\cap \forall L=\{1\}$ and $x\in F$, then $\forall x\in F\cap\forall L$, so $\forall x=1$, that is, $x=1$ (since Ker$(\forall)=\{1\}$), and so $F=\{1\}$. Therefore, $(L,\forall)$ is simple.

$(2)\Leftrightarrow(3)$  It is similar to the proof of Theorem 21 in \cite {Bezhanishvili}.

$(1)\Leftrightarrow(4)$ The equivalence of (1) and (4) follows from Definition 4.20.
\end{proof}

 Theorem 4.22 brings a method of how to check a monadic NM-algebra is simple. As an application of Theorem 4.22, one can check that the monadic NM-algebra in Example 4.21 is simple since $L_{\forall}=\{0,1\}$.

\begin{definition}\emph{A monadic NM-algebra $(L,\forall)$ is said to be\emph{ subdirectly irreducible} if it has the least nontrivial monadic congruence.}
\end{definition}

Let $(L,\forall)$ be a subdirectly irreducible monadic NM-algebra. Then, by Theorem 4.9, there exists a monadic filter $F$ of $(L,\forall)$ such that $\theta_F=F$, that means, $F$ is the least monadic filter of $(L,\forall)$ such that $F\neq\{1\}$. Thus, we can conclude that a monadic NM-algebra $(L,\forall)$ is said to be subdirectly irreducible if among the nontrivial monadic filters of $(L,\forall)$, there exists the least one, i.e., $\cap\{F\in MF(L)|F\neq\{1\}\}\neq\{1\}$.

\begin{example}\emph{Consider the monadic NM-algebra $(L,\forall)$ in Example 4.15. One can check that the set of monadic filters of $(L,\forall)$ are $\{e,f,g,1\}$, $\{e,g,1\}$, $\{e,f,1\}$, $\{1\}$ and hence $\cap\{F\in MF(L)|F\neq\{1\}\}=\cap\{\{e,f,g,1\},\{e,g,1\},\{e,f,1\}\}=\{e,1\}\neq\{1\}$. Therefore, $(L,\forall)$ is a subdirectly irreducible monadic NM-algebra. However, consider the monadic NM-algebra $(L,\forall)$ in Example 3.5, one can easily check that it is a monadic NM-algebra  but not a subdirectly irreducible monadic NM-algebra, since $\cap\{F\in MF(L)|F\neq\{1\}\}=\cap\{\{1,d\},\{b,c,1\}\}=\{1\}$.}
\end{example}

In the following, we will show that every subdirectly irreducible strong monadic NM-algebra is linearly ordered. For proving this important result, we need the following several propositions and theorems.
\begin{proposition} Let $(L,\forall)$ be a subdirectly irreducible monadic NM-algebra and $F_1$, $F_2$ be two monadic filters of $(L,\forall)$. If $F_1\cap F_2=\{1\}$, then $F_1=\{1\}$ or $F_2=\{1\}$.
\end{proposition}
\begin{proof} The proof is easy, and hence we omit the details.
\end{proof}

\begin{theorem} Let $(L,\forall)$ be a monadic NM-algebra. The the following statements are equivalent:
\begin{enumerate}[(1)]
  \item $(L,\forall)$ is a subdirectly irreducible monadic NM-algebra,
  \item there exists an element $a\in L$, $a<1$, such that for any $x\in L$, $x<1$, $a\in \langle x\rangle_\forall$.
\end{enumerate}
\end{theorem}
\begin{proof} $(1)\Rightarrow(2)$ If $(L,\forall)$ is a subdirectly irreducible monadic NM-algebra, then we have $\cap\{F\in MF(L)|F\neq\{1\}\}\neq\{1\}$, and hence $\cap\{\langle x\rangle_\forall|x<1\}\neq\{1\}$. If $a\in\cap\{\langle x\rangle_\forall|x<1\}$ satisfying $a\neq 1$, then we have  $x\neq 1$,$a\in \langle x\rangle_\forall$ for any $x\in L$, that is, there exists $m\in N$, such that $a\geq (\forall x)^m$. Clearly, $a$ is the element that we need.

$(2)\Rightarrow (1)$ In order to prove $(L,\forall)$ is a subdirectly irreducible monadic NM-algebra, we need to show that for any $F\in MF(L)$, if $F\neq\{1\}$, then $a\in F$. In fact, if $F\neq\{1\}$, then there exists $x\in F$, $x<1$. Further by (2), we have  $a\in \langle x\rangle_\forall$, and hence $a\in F$. So $a\in\cap\{F\in MF(L)|F\neq\{1\}\}$. Thus, $\cap\{F\in MF(L)|F\neq\{1\}\}\neq\{1\}$, that is, $(L,\forall)$ is a subdirectly irreducible monadic NM-algebra.
\end{proof}

 A non-unit element $a$ is said to be a \emph{coatom} of $L$  if  $a\leq b$, then $b\in \{a,1\}$, that is, $b=a$ or $b=1$ (see \cite{Blyth}). In the following proposition, we will show that every subdirectly irreducible strong monadic NM-algebra has at most one coatom.
\begin{proposition} Let $(L,\forall)$ be a subdirectly irreducible strong monadic NM-algebra. For any $x$,$y\in L$, if $x\vee y=1$, then $x=1$ or $y=1$.
\end{proposition}
\begin{proof} It follows from Proposition 4.25.
\end{proof}


The following theorem shows that the subdirectly irreducible strong NM-algebra $(L,\forall)$ is linearly ordered, that is to say, the truth value of all propositions in monadic NM-logic are comparable, this is the key importance from the logical point of view.
\begin{theorem} Let $(L,\forall)$ be a strong monadic NM-algebra. Then the following statements are equivalent:
\begin{enumerate}[(1)]
  \item $(L,\forall)$ is a subdirectly irreducible strong monadic NM-algebra,
  \item $(L,\forall)$ is a chain.
\end{enumerate}
\end{theorem}
\begin{proof} $(1)\Rightarrow(2)$ Suppose that $(L,\forall)$ is a subdirectly irreducible strong monadic NM-algebra. Applying Definition 2.1(4), we have $(x\rightarrow y)\vee(y\rightarrow x)=1$ for any $x,y\in L$, then by Proposition 4.27, we have $x\rightarrow y=1$ or $y\rightarrow x=1$, that is, $x\leq y$ or $y\leq x$. So $(L,\forall)$ is a chain.

$(2)\Rightarrow (1)$ Conversely, if $(L,\forall)$ is a nontrivial chain, then there exists a unique coatom, denoted by $a$. Suppose that $F$ is any monadic filter of $(L,\forall)$ satisfying $F\neq\{1\}$, then $a\in F$. Since $F$ is chosen arbitrarily from $MF(L)$, then $a\in\cap\{F\in MF(L)|F\neq\{1\}\}$. Hence $\cap\{F\in MF(L)|F\neq\{1\}\}\neq\{1\}$, i.e., $(L,\forall)$ is a subdirectly irreducible strong monadic NM-algebra.
\end{proof}

 In what follows, we will show that  $\forall L$ is linearly ordered when $(L,\forall)$ is a subdirectly irreducible monadic NM-algebra.

\begin{proposition} Let $(L,\forall)$ be a monadic NM-algebra. If $(L,\forall)$ is a subdirectly irreducible monadic NM-algebra, then $\forall L$ is linearly ordered.
\end{proposition}
\begin{proof} Let $F$ be the smallest non-trivial monadic filter of $(L,\forall)$ and $x\in F-\{1\}$. Assume that $\forall L$ is not linearly ordered, and $\forall a$, $\forall b \in \forall L$ such that $\forall a\nleqslant \forall b$ and $\forall b\nleqslant \forall a$. Then the filters $\langle \forall a\rightarrow \forall b\rangle_\forall$ and $\langle \forall b\rightarrow \forall a\rangle_\forall$ generated by $\forall a\rightarrow \forall b$ and $\forall b\rightarrow \forall a$ respectively, are non-trivial, and both contain $F$, in particular, $x\in \langle \forall a\rightarrow \forall b\rangle_\forall$ and $x\in\langle \forall b\rightarrow \forall a\rangle_\forall$. Since $\forall a\rightarrow \forall b\in \forall L$, by Theorem 4.4(1), there is a $n$ such that $x\geq (\forall(\forall a\rightarrow \forall b))^n\geq(\forall a\rightarrow \forall b)^n$ and $x\geq (\forall(\forall b\rightarrow \forall a))^n\geq(\forall b\rightarrow \forall a)^n$. Therefore, $x\geq (\forall a\rightarrow \forall b)^n\vee (\forall b\rightarrow \forall a)^n=1$ by Proposition 2.2(11). Thus, we can obtain that $x=1$, which is a contradiction the hypothesis.
\end{proof}

In the following, we will show that monadic NM-algebra $(L,\forall)$ is subdirectly irreducible if and only if NM-algebra $L_{\forall}$ is a subdirectly irreducible.
\begin{theorem} Let $(L,\forall)$ be a monadic NM-algebra. Then the following statements are equivalent:
\begin{enumerate}[(1)]
  \item $(L,\forall)$ is a subdirectly irreducible monadic NM-algebra,
  \item $L_{\forall}$ is a subdirectly irreducible subalgebra of $L$.
\end{enumerate}
\end{theorem}
\begin{proof}$(1)\Rightarrow(2)$ Let $(L,\forall)$ be a subdirectly irreducible monadic NM-algebra, then $MF[L]-\{1\}$ has a minimal element $F$. From Proposition 3.7(15), it is clear that $L_{\forall}$ is a subalgebra of $L$. Now, we will show that $F\cap L_{\forall}$ is the minimal monadic filter of $L_{\forall}$ such that $F\cap L_{\forall}\neq\{1\}$. First, if $F\cap L_{\forall}=\{1\}$, since $\forall F\subseteq F\cap L_{\forall}$ and hence $\forall x=1$ for any $x\in F$. Thus, $F\subseteq Ker(\forall)=\{1\}$ and $F=\{1\}$, which is a contradiction, that means, $F\cap L_{\forall}\neq\{1\}$. Next, we will show that $F\cap L_{\forall}$ is the minimal monadic filter of $(L,\forall)$. Suppose that $G$ is a filter of $L_{\forall}$, from corollary 4.6, we have $\langle G\rangle_\forall$ is the monadic filter of $(L,\forall)$ generated by $G$. Clearly $\langle G\rangle_\forall\cap L_{\forall}=G$. By minimality of $F$, $F\subseteq \langle G\rangle_\forall$ and hence $F\cap L_{\forall}\subseteq \langle G\rangle_\forall \cap L_{\forall}=G$. Then, $F\cap L_{\forall}$ is the minimal filter of $L_{\forall}$ such that $F\cap L_{\forall}\neq\{1\}$.  Therefore, $L_{\forall}$ is a subdirectly irreducible subalgebra of $L$.

$(2)\Rightarrow(1)$ Let $L_{\forall}$ is a subdirectly irreducible subalgebra of $L$. Then there exists a minimal filter $F$ of $L_{\forall}$ such that $F\neq\{1\}$. From Theorem 4.5, we get that $\langle F\cap L_\forall\rangle$ is a monadic filter of $(L,\forall)$. Further, we will show that $\langle F\cap L_\forall\rangle$ is a minimal monadic filter of $(L,\forall)$. In fact, if $G$ is another non-trivial monadic filter of $(L,\forall)$, then $G\cap L_{\forall}\supseteq F\cap L_{\forall}$. Then, $G$ contains the monadic filter generated by $F$, that is, $\langle F\cap L_\forall\rangle\subseteq\langle G\cap L_\forall\rangle$, i.e., $\langle F\cap L_\forall\rangle$ is minimal. Thus, $(L,\forall)$ is a subdirectly irreducible monadic NM-algebra.
\end{proof}

In the following, we will give another extension to monadic NM-algebras of a representation theorem.

\begin{proposition} Let $(L,\forall)$ be a strong monadic NM-algebra and $F$ be a monadic filter of $(L,\forall)$. Then the following properties hold: for any $x, y,z\in L$,
$$F=\langle F\cup \{x\rightarrow y\}\rangle_\forall\cap \langle F\cup \{y\rightarrow x\}\rangle_\forall$$
\end{proposition}
\begin{proof}It is easy to check that the forward inclusion is straightforward. Now, assume that $z$ is an element of both $\langle F\cup \{x\rightarrow y\}\rangle_\forall$ and $\langle F\cup \{y\rightarrow x\}\rangle_\forall$. Then, there are $f_1$,$f_2\in F$, $n_1,n_2\in N$ such that $f_1\odot (\forall(x\rightarrow y))^{n_1}\leq z$ and $f_2\odot (\forall(y\rightarrow x))^{n_2}\leq z$. If we let $f=f_1\odot f_2$ and $n=\max\{n_1,n_2\}$, it follows that $f\odot(\forall(x\rightarrow y))^{n}\leq z$ and $f\odot (\forall(y\rightarrow x))^{n}\leq z$. From the definition 2.1(3), we have  $(\forall(x\rightarrow y))^n\leq f\rightarrow z$ and $(\forall(y\rightarrow x))^n\leq f\rightarrow z$, and hence t $(\forall(x\rightarrow y))^n\vee (\forall(y\rightarrow x))^n\leq f\rightarrow z$. Further by Proposition 2.2(11), we have $(\forall(x\rightarrow y))^n\vee (\forall(y\rightarrow x))^n=(\forall((x\rightarrow y)\vee (y\rightarrow x)))^n=\forall 1=1$.  So $f\leq z$ and $z\in F$.
\end{proof}

\begin{proposition} Let $(L,\forall)$ be a strong monadic NM-algebra such that $\forall L$ is linearly ordered. Given $a\in L$, $a\neq 1$, there exists a prime monadic filter $P$ of $(L,\forall)$ such that $a\vee \forall r\notin P$ for any $r\neq 1$.
\end{proposition}
\begin{proof} Considering the set $C=\{a\vee \forall r| r\neq 1\}$. Note that $1\notin C$, since $a\vee\forall r=1$ implies that
$1=\forall (a \vee \forall r) = \forall a\vee \forall r$ and this would imply that $a = 1$ or $r = 1$.
Let $\mathcal{F}$ be the family of monadic filters $F$ in $(L,\forall)$ such that $F\cap C = \emptyset$. The above paragraph shows that
$\{1\}\in \mathcal{F}$, so that $\mathcal{F}$ is nonempty. In addition, it is straightforward to verify that any chain in $\mathcal{F}$
has an upper bound in $\mathcal{F}$. Hence, by Zorn's Lemma, there exists a maximal filter $P$ in $\mathcal{F}$.
We claim that $P$ is prime monadic filter. Indeed, let $x, y\in L$ and note that
$P = \langle P\cup\{x\rightarrow y\}\rangle_\forall \cap \langle P\cup\{y\rightarrow x\}\rangle_\forall$. If we assume that neither $\langle P\cup\{x\rightarrow y\}\rangle_\forall$ nor $\langle P\cup\{y\rightarrow x\}\rangle_\forall$ belongs to $\mathcal{F}$, then there are
$r_1, r_2\neq 1$ such that $a\vee \forall r_1 \in \langle P\cup\{x\rightarrow y\}\rangle_\forall$ and $a\vee \forall r_2 \in \langle P\cup\{y\rightarrow x\}\rangle_\forall$ . Since $\forall r_1$ and $\forall r_2$ are comparable, it follows that one of them belongs to both filters. Hence one of them belongs to $P$ , a contradiction. This shows that either $\langle P\cup\{x\rightarrow y\}\rangle_\forall\in \mathcal{F}$ or $\langle P\cup\{y \rightarrow x\}\rangle_\forall\in \mathcal{F}$. Assume the first option is true. By the maximality of $P$, we have $P = \langle P\cup\{x\rightarrow y\}\rangle_\forall$. Further by Proposition 3.7(5) and Theorem 4.4(2), we have $\forall x\rightarrow \forall y\in P$.
Analogously, $P = \langle P\cup\{y\rightarrow x\}\rangle_\forall$, so $\forall y\rightarrow \forall x\in P$. Thus, from Theorem 4.13 $(1)\Leftrightarrow (3)$, one can see that $P$ is a prime monadic filter of $(L,\forall)$.
\end{proof}
\begin{theorem} Let $(L,\forall)$ be a strong monadic NM-algebra. Then there exists a subdirect representation of the underlying NM-algebra $L\leq \prod_{i\in I}{L_i}$, where each $L_i$ is a totally ordered NM-algebra and $\forall L$ is embedded in $L_i$ via the corresponding projection map.
\end{theorem}
\begin{proof}For each $a\in L$, $a\neq1$, let $P_a$ be one of the prime filters provided by the previous proposition. Clearly $\cap _{a\neq 1}P_a =\{1\}$ and we obtain a natural embedding $L\rightarrow \prod_{a\neq1} L/P_a$. To close the proof
we need only show that the natural map $L\rightarrow L/P_a$ is injective on $\forall L$. Indeed, suppose there
were $r_1, r_2\in L$ such that $\forall r_1 \leq \forall r_2$ and $\forall r_1/P_a = \forall r_2/P_a$. We have that $\forall r_2\rightarrow \forall r_1=\forall(\forall r_2\rightarrow \forall r_1)$ and $\forall r_2\rightarrow \forall r_1\neq 1$. Hence, we know that $a\vee (\forall r_2\rightarrow \forall r_1)\notin P_a$, which is a contradiction.
\end{proof}
\section{Monadic NM-logic}
In this section, we present the monadic NM-logic (MNL for short) and prove the (chain) completeness of this logic with respect to the variety of (strong) monadic NM-algebras.\\

The language of MNL consists of countably many proposition variables $(p_1,p_2,\cdots)$, the constants \={0} and \={1}, the unary logic connective $\forall$, the binary logic connectives $\sqcap$, $\sqcup$, \&, $\Rightarrow$ and finally the auxiliary symbols ``(and)". Formulas are defined inductively: proposition variables, \={0} and \={1} are formulas; if $\phi$ and $\psi$ are formulas, then so are $\phi\sqcap\psi$, $\phi\sqcup\psi$, $\phi\&\psi$, $\phi\Rightarrow \psi$, $\forall \phi$. One useful shorthand notations denoted by: \={1} for \={0}$\Rightarrow$ \={0} and $\phi\equiv \psi$ for $(\phi\Rightarrow \psi)\sqcap (\psi\Rightarrow \phi)$  for any formula $\phi$ and $\psi$.\\

Adapting for the axiomatization of monadic predicate NM logic gave by Bianchi in \cite{Bianchi}, we can define monadic NM logic (MNL for short) as a logic which contains NM-logic, the formulas as the axioms schemes:
\\
(MTL1) $(\phi\Rightarrow \psi)\Rightarrow((\psi\Rightarrow\chi)\Rightarrow(\phi\Rightarrow\chi))$,\\
(MTL2) $(\phi\&\psi)\Rightarrow \phi$,\\
(MTL3) $(\phi\&\psi)\Rightarrow(\psi\&\phi)$,\\
(MTL4) $(\phi\sqcap \psi)\Rightarrow \phi$,\\
(MTL5) $(\phi\sqcap \psi)\Rightarrow (\psi\sqcap \phi)$,\\
(MTL6) $(\phi\&(\phi\Rightarrow \psi)\Rightarrow(\phi\sqcap\psi)$,\\
(MTL7a) $(\phi\Rightarrow(\psi\Rightarrow\chi))\Rightarrow((\phi\&\psi)\Rightarrow\chi)$,\\
(MTL7b) $((\phi\&\psi)\Rightarrow\chi)\Rightarrow(\phi\Rightarrow(\psi\Rightarrow \chi))$,\\
(MTL8) $((\phi\Rightarrow\psi)\Rightarrow\chi)\Rightarrow(((\psi\Rightarrow \phi)\Rightarrow \chi)\Rightarrow \chi)$,\\
(MTL9) \={0}$\Rightarrow\phi $,\\
(DN) $(\phi\Rightarrow $\={0})$ \Rightarrow $\={0}$ \Rightarrow\phi$,\\
(WNM) $((\phi\Rightarrow$ \={0})$\&\psi)\vee ((\phi\sqcap \psi)\Rightarrow(\phi\&\psi))$,\\
(U1) $\forall \phi\Rightarrow\phi$,\\
(U2) $\forall((\phi\Rightarrow $\={0}$)\Rightarrow \psi)\Rightarrow ((\forall \phi)\Rightarrow $\={0}$)\Rightarrow \forall \psi$,\\
(U3) $\forall(\forall\phi\Rightarrow  \psi)\Rightarrow \forall \phi\Rightarrow \forall \psi$,\\
(U4) $\forall(\phi\sqcup \forall\psi)\Rightarrow \forall\phi\sqcup \forall\psi $.\\
and is closed under modus ponens and necessitation $(\phi/\forall \phi)$.\\

As well known, monadic Boolean algebras serve as algebraic models of the one-variable fragment of the classical predicate calculus CPC \cite{Halmos}, while monadic Heyting algebras serve the same purpose for the one-variable fragment of intuitionistic predicate calculus IPC \cite{Bezhanishvili}, where $S5$ modal logic in the case of CPC and MIPC (Prior's intuitionistic modal logic) in the case of IPC. For monadic NM-algebras, which serve some purpose for the one-variable fragment of NM predicate calculus NMPC \cite{Bianchi}, it is routine to check that $S5(\mathbf{NM})$ modal fuzzy logic in the case of NMPC and $S5$ modal logic in the case of CPC, MIPC in the case of IPC. Moreover, H$\acute{a}$jek introduced  basic predicate logic BPC and $S5(\mathbf{BL})$ modal fuzzy logic and proved that $S5(\mathbf{BL})$ modal fuzzy logic is equivalent to monadic basic logic in \cite{Hajek}, that is, monadic basic logic axiomatizes the one variable fragment of the basic predicate logic. By the similarly way, we can prove that $S5(\mathbf{NM})$ modal fuzzy logic is equivalent to monadic  NM-logic (in fact, monadic NM-logic stands for $S5(\mathbf{NM})$ modal fuzzy logic, which contains two modal connectives $\Box$ and $\diamondsuit$, where $\forall$ stands to $\Box$ just as $\exists$ stands to $\diamondsuit$ and
$\exists\phi\equiv (\forall(\phi\Rightarrow $\={0}))$ \Rightarrow $\={0}$ $).  Thus, monadic NM-logic axiomatizes the one variable fragment of the predicate NM-logic.\\


 Let $N$ denote a first-order language based on $\wedge,\vee,\odot,\rightarrow,\forall$ and $N_m$ denote the monadic propositional language based on $\sqcap,\sqcup,\&,\Rightarrow,\forall$, and $Form(N)$ and $Form(N_m)$ be the set of all formulas of $N$, $N_m$, respectively. We fixed a variable $x$ in $N$ , associate with each propositional letter $p$ in $Form(N_m)$ a unique monadic predicate $p^\bullet(x)$ in $Form(N)$ and defined by induction a translation $\Psi:Form(N_m)\longrightarrow Form(N)$ by putting:
 \begin{enumerate}[(1)]
   \item $\Psi(p)=p^\bullet(x)$ if $p$ is propositional variable,
   \item $\Psi(\alpha\circ\beta)=\Psi(\alpha)\circ \Psi(\beta)$, where $\circ=\wedge,\vee,\odot,\rightarrow$,
   \item $\Psi(\forall\alpha)=\forall x\Psi(\alpha)$.
  \end{enumerate}

Through this translation $\Psi$, we can identity the formulas of $N_m$ with the monadic formulas of $N$ containing the variable $x$. Moreover, it is routine to check that $\Psi(MNL)\subseteq NMPC$, where $NMPC$ stand for the predicate calculus of NM-logic that is defined in the paper \cite{Bianchi}.\\

In order to prove a completeness theorem, we are going to summarize some necessary notions of of MNL, which will be used in the further.

The consequence relation $\vdash$ is defined in the usual way. Let $T$ be a theory, i.e., a set of formulas in MNL. A (formula) proof of a formula $\phi$ in $T$ is a finite sequence of formulas with $\phi$ at its end, such that every formula in the sequence is either an axiom of MNL, a formula of $T$, or the result of an application of an deduction rule to previous formulas in the sequence. If a proof of $\phi$ exists in $T$, we say that $\phi$ can be deduced from $T$ and denote this by $T\vdash \phi$. Moreover $T$ is complete if for each pair $\phi$, $\psi$, $T\vdash \phi \Rightarrow \psi$ or $T\vdash \psi \Rightarrow \phi$.

\begin{definition} \emph{Let $\mathcal{L}=(L,\forall)$ be a monadic NM-algebra and $T$ be a theory. An $\mathcal{L}$-evaluation is a mapping $e$ from the set of formulas of MNL to $L$ that satisfies, for each two formulas $\phi$ and $\psi$: $e(\phi\sqcap \psi)=e(\phi)\wedge e(\psi)$, $e(\phi\sqcup \psi)=e(\phi)\vee e(\psi)$, $e(\phi\Rightarrow \psi)=e(\phi)\rightarrow e(\psi)$, $e(\phi\& \psi)=e(\phi)\odot e(\psi)$, $e(\forall \phi)=\forall e(\phi)$, $e$(\={0})=0 and $e$(\={1})=1. If a $\mathcal{L}$-evaluation $e$ satisfies $e(\chi)=1$ for every $\chi$ in $T$, it is called a $\mathcal{L}$-model of $T$.}
\end{definition}

Now, we stress our attention to the Lindenbaum-Tarski algebra of MNL.
\begin{definition}\emph { Let $T$ be a fixed theory over MNL. For each formula $\phi$, let $[\phi]_T$ be the set of all formulas $\psi$ such that $T\vdash \phi\equiv\psi$ and $L/T$ be the set of all the class $[\phi]_T$.We define:
$0=[0]_T$, $1=[1]_T$,
$[\phi]_T\rightarrow [\psi]_T=[\phi\Rightarrow \psi]_T$,
$[\phi]_T\vee [\psi]_T=[\phi\sqcup \psi]_T$,
$[\phi]_T\wedge [\psi]_T=[\phi\sqcap \psi]_T$,
$[\phi]_T\odot[\psi]_T=[\phi\& \psi]_T$,
$\forall_T{[\phi]_T}=[\forall\phi]_T$.
This algebra is denoted by $\mathcal{L_T}=(L/T,\wedge,\vee,\odot,\rightarrow,0,1,\forall_T)$.}
\end{definition}
\begin{proposition} $\mathcal{L_T}=(L/T,\wedge,\vee,\odot,\rightarrow,0,1,\forall_T)$ is a monadic NM-algebra.
\end{proposition}
\begin{proof} It is similar to the proof Proposition 4.10.
\end{proof}

\begin{theorem} Let $T$ be a theory over MNL. Then $T$ is complete if and only if the monadic NM-algebra $\mathcal{L_T}=(L/T,\wedge,\vee,\odot,\rightarrow,0,1,\forall_T)$ is linearly ordered.
\end{theorem}
\begin{proof} It is similar to the proof of Theorem 4.13 $(1)\Leftrightarrow (4)$.
\end{proof}

It is easy to check that MNL is sound with respect to the variety of monadic NM-algebras, i.e., that is, if a formula $\phi$ can be deduced from a theory $T$ in MNL, then for every monadic NM-algebra $\mathcal{L}$ and for every $\mathcal{L}$-model $e$ of $T$, $e(\phi)=1$. Indeed, we need to verify the soundness of the new axioms and deduction of MNL (for the axioms and rules of NM, the reader can check \cite{Esteva}). For the axioms this is easy, as they are straightforward generalizations of axioms of monadic NM-algebras. We will now verify the soundness of the new deduction rules.
\begin{proposition} The deduction rules of MNL are sound in the following sense, for any formula $\phi$ and $\psi$.
\begin{enumerate}[(1)]
  \item If for all monadic NM-algebra $\mathcal{L}$ and for all $\mathcal{L}$-model $e$ for $T$, $e(\phi)=1$, then for all monadic NM-algebra $\mathcal{L}$ and for all $\mathcal{L}$-tautology $e$ for $T$, $e(\forall\phi)=1$.
  \item If for all monadic NM-algebra $\mathcal{L}$ and for all $\mathcal{L}$-model $e$ for $T$, $e(\phi)=1$ and $e(\phi\rightarrow \psi)=1$, then for all monadic NM-algebra $\mathcal{L}$ and for all $\mathcal{L}$-model $e$ for $T$, $e(\psi)=1$.
\end{enumerate}
\end{proposition}
\begin{proof} $(1)$ It follows from Proposition 3.7(2) and Definition 5.1.

$(2)$ It follows from Proposition 2.2(1) and Definition 5.1.
\end{proof}

In the following, we prove a completeness theorem of MNL based on monadic NM-algebras.

\begin{theorem}Let $T$ be a theory over MNL. For each formula $\phi$, the following statements are equivalent:
\begin{enumerate}[(1)]
  \item $T\vdash \phi$,
  \item for each  monadic NM-algebra $\mathcal{L}$ and for every $\mathcal{L}$-model $e$ of $T$, $e(\phi)=1$ .
\end{enumerate}
\end{theorem}
\begin{proof} $(1)\Rightarrow (2)$ It follows from Propositions 5.3 and 5.5.

$(2)\Rightarrow (1)$  To this end recall Proposition 5.3 saying, among other things, that the
algebra $\mathcal{L}_{T}$ of classes of equivalent formulas of monadic NM-logic, is a monadic
NM-algebra, thus $\phi$ is a $\mathcal{L}_{T}$ tautology if it satisfy (2). In particular, let $e(p_i) =
[p_i]_{T}$ and $e(\forall p_i) = [\forall p_i]_{T}$ for all propositional variables pi. Then $e(\phi) =
[\phi]_{T} = [1]_{T}$, thus $T\vdash\phi\equiv 1$, hence $T\vdash \phi$.
\end{proof}

In what follows, we shall analyse one axiomatic extension of MNL in order to prove completeness with respect to linearly ordered structures. The corresponding classes of models are strong monadic NM-algebras, that is a subvarieties of monadic NM-algebras. First, we introduce the propositional calculus strong monadic NM-logic (SMNL for short), which is an axiomatic extension of MNL.

\begin{definition}\emph{The axioms of SMNL are those of MNL plus  $\forall(\phi\sqcup\psi)\rightarrow \forall\phi\sqcup \forall\psi$, where $\phi$ and $\psi$ are arbitrary formulas of MNL.}
\end{definition}

If $T$ is an arbitrary set of formulas then the Lindenbaum-Tarski algebra is defined as usual and it will be denoted by  $\mathcal{SL_T}$. It is obvious that $\mathcal{SL_T}$ is a monadic NM-algebra which satisfies the algebraic identities corresponding to the logical axioms $\forall(\phi\sqcup\psi)\rightarrow \forall\phi\sqcup \forall\psi$.

\begin{theorem}Let $T$ be a theory over SMNL and $T\nvdash \phi$. Then there is a consistent complete supertheory $T^{\prime}\supseteq T$ such that $T^{\prime}\nvdash\phi$.
\end{theorem}
\begin{proof} It is similar to the proof of Theorem 4.17.
\end{proof}

  One can check that Lindenbaum-Tarski algebra $\mathcal{SL_T}$ is a strong monadic NM-algebra for any set $T$ of formulas. Hence, the semantics of SMNL uses evaluations with values in strong monadic NM-algebra. Further by Theorem 4.18, the following completeness result is straightforward.

\begin{theorem} Let $T$ be a theory over SMNL. For each formula $\phi$, the following statements are equivalent:
\begin{enumerate}[(1)]
  \item $T\vdash \phi$,
  \item for each  strong monadic NM-algebra $\mathcal{L}$ and for every $\mathcal{L}$-model $e$ of $T$, $e(\phi)=1$,
  \item for each linearly ordered strong monadic NM-algebra $\mathcal{L}$ and for every $\mathcal{L}$-model $e$ of $T$, $e(\phi)=1$.
\end{enumerate}
\end{theorem}
\begin{proof}It follows from Theorems 4.18, 5.8.
\end{proof}
\section{Conclusions}
 Motivated by the previous research of monadic MV-algebras, we introduced and investigated monadic NM-algebras. Then, we discuss the relations between monadic NM-algebras and some related structures. Also, we characterize two kinds of monadic NM-algebras and obatin a monadic analogous of representation theorem for NM-algebras. Finally, we introduce the monadic NM-logic and prove the (chain) completeness of monadic NM-logic.  Since the above topics are of current interest, we suggest further directions of research:
\begin{enumerate}[(1)]
  \item Introducing and studying polyadic NM-algebras, which are further generalizations of monadic NM-algebras given by polyadic structures.
  \item Focusing on the varieties of monadic NM-algebras. In particular, one can investigate semisimple, locally finite, finitely approximated and splitting varieties of monadic NM-algebras as well as varieties with the disjunction and the existence properties.
\end{enumerate}

\medskip
\noindent\textbf{Acknowledgments}
\medskip

\indent The authors are extremely grateful to the editor
and the referees for their valuable comments and helpful suggestions
which help to improve the presentation of this paper.
This study was funded by a grant of National Natural Science Foundation of China (11571281,11601302), Postdoctoral Science Foundation of China
(2016M602761) and Natural Science Foundation of Shaanxi Province (2017JQ1005).


\end{document}